\documentclass[a4paper,12pt]{amsart}

\usepackage[ngerman, italian, english]{babel}
\usepackage{amssymb}
\usepackage{hyperref}
\usepackage{mathptmx}
\usepackage{verbatim}
\usepackage[utf8]{inputenc}

\theoremstyle{plain}
\newtheorem{thm}{\bf Theorem}[section]

\newtheorem{prop}[thm]{\bf Proposition}
\newtheorem{lemma}[thm]{\bf Lemma}

\theoremstyle{definition}
\newtheorem{definition}[thm]{\bf Definition}
\theoremstyle{remark}

\theoremstyle{example}
\numberwithin{equation}{section}

\def\mult{\operatorname{e}}

\def\GL{\operatorname{GL}}

\def \Ker{{\operatorname{Ker}}}
\def \Hom{{\operatorname{Hom}}}

\def \Sym{{\operatorname{Sym}}}

\def \tl{{}^{\textup{t}\negthinspace}}

\def\Bw{\bigwedge}
\def\Tensor{\bigotimes}
\def\tensor{\otimes}

\let\iso=\cong
\def\YY{{\mathbb Y}}

\begin{document}

\title{Partitions of single exterior type} \thanks{The authors thank the
Mathematical Sciences Research Institute, Berkeley CA, where
this work started to take shape, for support and hospitality
during Fall 2012}
\author{Winfried Bruns}
\address{Universit\"at Osnabr\"uck, Institut f\"ur Mathematik, 49069 Osnabr\"uck, Germany}
\email{wbruns@uos.de}
\author{Matteo Varbaro}
\address{Dipartimento di Matematica,
Universit\`a degli Studi di Genova, Italy}
\email{varbaro@dima.unige.it}
\date{}
\subjclass[2000]{13A50, 14M12, 14L30} \keywords{Relations of
minors; Determinantal varieties; Plethysms} \maketitle

\begin{abstract}
We characterize the irreducible representations of the general
linear group $\GL(V)$ that have multiplicity $1$ in the direct
sum of all Schur modules of a given exterior power of $V$.
These have come up in connection with the relations of the
lower order minors of a generic matrix. We show that the
minimal relations conjectured by Bruns, Conca and Varbaro are
exactly those coming from partitions of single exterior type.
\end{abstract}

\section{Introduction}

The main motivation for this note was the desire to provide
further evidence for a conjecture of Conca and the authors
\cite[Conjecture 2.12]{BCV} on the polynomial relations between
the $t$-minors of a generic matrix. With the notation in
\cite{BCV}, let $X=(x_{ij})$ denote an $m\times n$ matrix of
indeterminates over a field $K$ of characteristic $0$, $R=K[X]$
the polynomial ring over the variables $x_{ij}$ and
$A_t\subseteq R$ the $K$-subalgebra of $R$ generated by the
$t$-minors of $X$. With respect to a choice of bases in
$K$-vector spaces $V$ and $W$ of dimension $m$ and $n$,
respectively, one has a natural action of the group
$G=\GL(V)\times \GL(W)$ on $R$, induced by
$$
(A,B)\cdot
X=AXB^{-1} \qquad\forall \ A\in\GL(V),B\in\GL(W).
$$
This action restricts to $A_t$, making $A_t$ a $G$-algebra.
Since the $G$-decomposition of $A_t$ can be deduced from the
work of De Concini, Eisenbud and Procesi \cite{DEP}, it is
natural to exploit such an action. A presentation of $A_t$ as a
quotient of a polynomial ring is provided by the natural
projection
$$
\pi:S_t\rightarrow A_t,
$$
where $S_t=\Sym(\Bw^tV\tensor\Bw^tW^*)$. Also $S_t$ is a
$G$-algebra, and the map $\pi$ is $G$-equivariant. Therefore
the ideal of relations $J_t=\Ker(\pi)$ is a $G$-module as well.

The conjecture \cite[Conjecture 2.12]{BCV} predicts a minimal
list of irreducible $G$-modules generating $J_t$, or, by
Nakayama's lemma, the decomposition of
$$
J_t\tensor_{S_t}K
$$
where we identify $K$ and the residue class field of $R$ with
respect to the irrelevant maximal ideal generated by the
indeterminates. In particular, the conjecture predicts that
$J_t$ is generated in degrees $2$ and $3$.

In the assignment of partitions to Young diagrams and to an
irreducible representation of $\GL(V)$ we follow Weyman
\cite{We}: a partition of nonnegative integers
$\lambda=(\lambda_1,\dots,\lambda_k)$,
$\lambda_1\ge\dots\ge\lambda_k$, is pictorially represented by
$k$ rows of boxes of lengths $\lambda_1,\dots,\lambda_k$ with
coordinates in the fourth quadrant, and a single row of length
$m$ represents $\Bw^m V$. The highest weight of the
representation is then given by the transpose partition
$\tl\lambda$ in which rows and columns are exchanged:
$(\tl\lambda)_i=|\{j:\lambda_j\ge i\}|$. With this convention,
we denote the Schur module associated with the partition
$\lambda$ and the vector space $V$ by $L_\lambda V$.

Because $S_t$ is a quotient of
$$
T_t=\bigoplus_{d\geq
0}\biggl(\Tensor^d\Bigl(\Bw^tV \ \tensor \
\Bw^tW^*\Bigr)\biggr)=\bigoplus_{d\geq
0}\biggl(\Tensor^d\Bw^tV \ \tensor \
\Tensor^d\Bw^tW^*\biggr),
$$
Pieri's rule implies that the irreducible summands of
$J_t\tensor_{S_t}K$ must be of the form
$$
L_{\gamma}V\tensor L_{\lambda}W^*
$$
where $\gamma$ and $\lambda$ are partitions satisfying the
following conditions:
\begin{itemize}
\item[(i)] $\gamma,\lambda\vdash dt$,
\item[(ii)] both $\gamma$ and $\lambda$ have at most $d$
    rows.
\end{itemize}
We call such partitions (or \emph{bipartitions}
$(\gamma|\lambda)$) \emph{$(t,d)$-admissible} (just
$t$-admissible if we do not need to emphasize the degree). In
\cite{BCV} a set $A$ of $(t,2)$-admissible bipartitions
$(\gamma|\lambda)$ and a set $B$ of $(t,3)$-admissible
bipartitions $(\gamma|\lambda)$ were found such that
\begin{equation}
\bigoplus_{(\gamma|\lambda)\in A}L_{\gamma}V\tensor
L_{\lambda}W^* \ \oplus \ \bigoplus_{(\gamma|\lambda)\in
B}L_{\gamma}V\tensor L_{\lambda}W^*\subseteq
J_t\tensor_{S_t}K. \label{contains}
\end{equation}
Conjecture 2.12 in \cite{BCV} states that the inclusion in
Equation \eqref{contains} is an equality. For the convenience
of the reader and since it is crucial for the following  we
recall how $A$ and $B$ are defined.
\begin{itemize}
\item[(i)] For $u\in\{0,\ldots ,t\}$ let:
$$\tau_u=(t+u,t-u).$$
\item[(ii)] For $u\in\{1,\ldots ,\lfloor t/2\rfloor\}$ let
$$\gamma_u = (t+u, \ t+u, \ t-2u) \qquad\mbox{and} \qquad\lambda_u=(t+2u, \ t-u, \ t-u).$$
\item[(iii)] For each $u\in\{2,\ldots ,\lceil t/2\rceil\}$
    let
$$\rho_u = (t+u, \ t+u-1, \ t-2u+1) \qquad\mbox{and} \qquad\sigma_u=(t+2u-1, \ t-u+1, \ t-u).$$
\end{itemize}
With this notation,
\begin{align*}
A&=\bigl\{(\tau_u|\tau_v): 0\le u,v\le t,\  u+v \mbox{ even}, u\neq v\bigr\},\\
B&=\bigl\{(\gamma_u|\lambda_u),(\lambda_u|\gamma_u):1\le u \le \lfloor t/2\rfloor\bigr\}\\
&\qquad\qquad\qquad\qquad\qquad\cup\bigl\{((\rho_v|\sigma_v),(\sigma_v|\rho_v):2\le v\le\lceil
t/2\rceil\bigr\}.
\end{align*}

Note that not all the partitions above are supported by the
underlying vector spaces if their dimensions are too small: a
partition $\lambda$ can only appear in a representation of
$\GL(V)$ if $\lambda_1\le \dim V$. For simplicity we have
passed this point over since it is essentially irrelevant. The
reader is advised to remove all partitions from the statements
that are too large for the vector spaces under consideration.

The decomposition of $S_t$ as a module over the ``big'' group
$$
H=\GL(E)\times\GL(F),\qquad E=\Bw^t V,\quad F=\Bw^t W,
$$
is well known by Cauchy's rule:
\begin{equation}
S_t=\bigoplus_\mu L_\mu E\tensor L_\mu F^* \label{Cauchy}
\end{equation}
where $\mu$ is extended over all partitions. The
$\GL(V)$-decomposition of $L_\mu E$ is an essentially unsolved
plethysm. However, the partitions in the definition of $A$ and
$B$ play a very special role in it, as was already observed in
\cite{BCV}:

\begin{definition}
Let $\lambda\vdash dt$ be $t$-admissible. Then $\lambda$ is
said to be of \emph{single $\Bw^t$-type} $\mu$ if $\mu\vdash d$
is the only partition such that $L_\lambda V$ is a direct
summand of $L_\mu (\Bw^tV)$ and, moreover, has multiplicity $1$
in it. Without specifying $\mu$, notice that $\lambda$ is of
single $\Bw^t$-type if and only if $\lambda$ has multiplicity
$1$ in $\bigoplus_{\alpha\vdash d} L_{\alpha} (\Bw^t V)$.
\end{definition}

In this note we will classify all partitions of single
$\Bw^t$-type (or simply \emph{single exterior type}) and show
that the bi-partitions in the sets $A$ and $B$ are exactly
those of single $\Bw^t$-type that occur in a minimal generating
set of $J_t$. While this observation does certainly not prove
the conjecture in \cite{BCV}, it provides further evidence for
it.

\section{Auxiliary results on partitions}

In this section we discuss two transformations of partitions
that preserve single exterior type. It was already observed in
\cite{BCV} that trivial extensions in the following sense are
irrelevant: if a partition $\tilde \lambda$ arises from a
$t$-admissible partition $\lambda\vdash dt$ by prefixing
$\lambda$ with columns of length $d$, then $\tilde\lambda$ is
called a \emph{trivial extension} of $\lambda$. We quote
\cite[1.16]{BCV} ($\mult_\lambda$ denotes the multiplicity of
$\lambda$):

\begin{prop}\label{propretract}
Let $\mu$ be a partition of $d$ and consider partitions
$\lambda =(\lambda_1,\dots,\lambda_k)\vdash td$ with $k\leq d$
and
$\tilde{\lambda}=(\lambda_1+1,\dots,\lambda_k+1,1,\dots,1)\vdash
dt+d$. If $\dim_{K}V\geq \lambda_1+1$, then
$$
\mult_\lambda \Bigl(L_\mu \Bigl( \Bw^tV\Bigr)\Bigr) =
\mult_{\tilde{\lambda}}\Bigl(L_\mu\Bigl(\Bw^{t+1}V\Bigr)\Bigr).
$$
In particular, $\lambda$ is of single $\Bw^t$-type $\mu$ if and
only if $\tilde\lambda$ is of single $\Bw^{t+1}$-type $\mu$.
\end{prop}

Next we want to show that a similar result holds for
dualization, in the sense that $\Bw^{n-t} V$, $n=\dim V$, is
dual to $\Bw^t V$ (up to tensoring with the determinant). Let
$\lambda =(\lambda_1,\ldots ,\lambda_k)\vdash td$ be
$t$-admissible; then we set
$$
\lambda^{*,n}=(n-\lambda_d,\ldots ,n-\lambda_1)\vdash (n-t)d.
$$
Evidently $\lambda^{*,n}$ is $(n-t)$-admissible. Note that
$\lambda$ and $\lambda^{*,n}$ rotated by $180^\circ$ degrees
complement each other to a $d\times n$ rectangle (representing
the $d$-th tensor power of the determinant $\det V= \Bw^n V$
when $n=\dim V$).

Notice that $\lambda^{*,n}$ is a trivial extension of
$\lambda^{*,\lambda_1}$. In view of this we will denote
$\lambda^{*,\lambda_1}$ just with $\lambda^{*}$, calling it
simply the \emph{dual} of $\lambda$. Also, note that if $k=d$,
so that $\lambda$ is a trivial extension of some $\gamma$, then
$\lambda^{*,n}=\gamma^{*,n}$. Therefore, when speaking of dual
partitions, we will usually assume that $n=\lambda_1$
and $k<d$.

\begin{prop}\label{dualization}
Let $\mu$ be a partition of $d$ and consider a $t$-admissible
partition $\lambda =(\lambda_1,\dots,\lambda_k)\vdash td$.
Suppose $\dim V=\lambda_1$. Then
$$
\mult_\lambda \Bigl(L_\mu\Bw^t V\Bigr) =  \mult_{\lambda^*}\bigl(L_\mu(\Bw^{\lambda_1-t} V\bigr).
$$
In particular, $\lambda$ is of single $\Bw^t$-type if an only
if $\lambda^*$ is of single $\Bw^{\lambda_1-t}$-type.
\end{prop}

\begin{proof}
Set $n=\dim V=\lambda_1$. Consider the $\GL(V)$-equivariant
multiplication
$$
\Bw^t V\tensor\Bw^{n-t}V\to \det V.
$$
It induces an equivariant isomorphism
$$
\Bw^{t} V\iso \Hom_K\Bigl(\Bw^{n-t}V,\det V\Bigr)=\Bigl(\Bw^{n-t}V\Bigr)^*\tensor
\det V=\Bigl(\Bw^{n-t}V^*\Bigr)\tensor\det V.
$$
Next we can pass to the $d$-th tensor power on the right and
the left, and apply the Young symmetrizer $\YY_\mu$ (see Fulton
and Harris \cite[p. 46]{FH} inverting rows and columns) to
obtain a $\GL(V)$-equivariant isomorphism
$$
\YY_\mu \Tensor^d\Bw^tV\iso \YY_\mu \Tensor^d\Bigl(\Bw^{n-t} V^*\tensor \det V\Bigr).
$$
Next we can go from $\YY_\mu \Tensor^d(\Bw^{n-t} V^*\tensor
\det V)$ to $\YY_\mu \Tensor^d\Bw^{n-t} V^*$, except that we
have to subtract the weight of $\Tensor^d \det V$ from each
weight in $\YY_\mu \Tensor^d(\Bw^{n-t} V^*\tensor \det V)$.
Finally, if we replace $\GL(V)$ by $\GL(V^*)$ as the acting
group, we see that every partition $\lambda$ in $\YY_\mu
\Tensor^d \Bw^t V$ goes with equal multiplicity to the
partition $\lambda^*$ in $\YY_\mu \Tensor^d\Bw^{n-t} V^*$. But
the multiplicities depend only on the dimension of the basic
vector space, and therefore we can replace $\Bw^{n-t} V^*$ by
$\Bw^{n-t} V$.
\end{proof}

Below we will use the obvious generalization of Proposition
\ref{dualization} to $\lambda^{*,n}$ that results from
Proposition \ref{propretract}.

\section{Partitions of single exterior type}

The characterization of partitions of single exterior type is
based on a recursive criterion established in \cite{BCV}. For
it and also for the characterization of the minimal relations
of single exterior type we need the same terminology.

Let $\lambda$ be a $(t,d)$-admissible diagram.  Given $1\leq
e\leq d$, we say that $\alpha$ is a \emph{$(t,e)$-predecessor}
of $\lambda$ if and only if $\alpha$ is a $(t,d-e)$-admissible
diagram such that $\tl\alpha_i\leq \tl\lambda_i\leq
\tl\alpha_i+e$ for all $i=1,\ldots ,\lambda_1$ (we set
$\tl\alpha_i=0$ if $i>\alpha_1$). In such a case we also say
that $\lambda$ is a \emph{$(t,e)$-successor} of $\alpha$. If we
just say that $\alpha$ is a \emph{$t$-predecessor} of
$\lambda$, we mean that $\alpha$ is a $(t,e)$-predecessor of
$\lambda$ for some $e$, and analogously for $\lambda$ being a
\emph{$t$-successor} of $\alpha$. (This terminology deviates
slightly from \cite{BCV} where a predecessor is necessarily a
$(t,1)$-predecessor.) The Littlewood-Richardson rule implies at
once that, for a $(t,d)$-admissible diagram $\lambda$ and a
$(t,d-e)$-admissible diagram $\alpha$ the following are
equivalent:
\begin{itemize}
\item[(i)] $\alpha$ is a $(t,e)$-predecessor of $\lambda$.
\item[(ii)] $L_{\lambda}V$ occurs in $(\Tensor^e \Bw^{t}V)
    \tensor L_{\alpha}V$, where $V$ is a
$K$-vector space of dimension $\geq \lambda_1$.
\end{itemize}

Now we can quote the following criterion for single
$\Bw^t$-type from \cite[Proposition 1.22]{BCV}. (Condition (iv)
has been added here. It strengthens (iii), but follows from
(iii) by induction.)

\begin{prop}\label{single_char}
Let $\lambda\vdash dt$ and $\mu\vdash d$ be partitions such
that $L_\lambda V$ occurs in $L_\mu (\Bw^tV)$. Then the
following are equivalent:
\begin{enumerate}
\item[(i)] $\lambda$ is of single $\Bw^t$-type;
\item[(ii)] the multiplicities of $\lambda$ and of $\mu$ in
    $\Tensor^d(\Bw^t V)$ coincide;
\item[(iii)] every $(t,1)$-predecessor $\lambda'$ of
    $\lambda$ is of single $\Bw^t$-type $\mu'$ where $\mu'$
    is a $(1,1)$-predecessor of $\mu$, and no two distinct
    $(t,1)$-predecessors of $\lambda$ share the same
    $(1,1)$-predecessor $\mu'$ of $\mu$;
\item[(iv)] every $t$-predecessor $\lambda'$ of $\lambda$
    is of single $\Bw^t$-type $\mu'$ where $\mu'$ is a
    $1$-predecessor of $\mu$, and no two distinct
    $t$-predecessors of $\lambda$ share the same
    $1$-predecessor $\mu'$ of $\mu$.
\end{enumerate}
\end{prop}

As we will see in a moment, one class of single $\Bw^t$-type
partitions is given by the hooks.

\begin{definition}
A diagram $\lambda=(\lambda_1,\ldots ,\lambda_k)$ with
$\lambda_2\leq 1$ is called a \emph{hook}.
\end{definition}

A hook can be always written like $(a,1^b)$, where $1^b$ means
$b$ ones.

\begin{lemma}\label{hooklemma}
Let $d>0$ and $k\in\{0,\ldots ,d-1\}$. Then $(td-k,1^k)$ is of
single $\Bw^t$-type $\mu$ where:
\begin{itemize}
\item[{(i)}] $\mu=(d-k,1^k)$ if $t$ is odd.
\item[{(ii)}] $\mu=(k+1,1^{d-k-1})$ if $t$ is even.
\end{itemize}
\end{lemma}
\begin{proof}
Let us fix $t$ and use induction on $d$. For $d=2$ the
statement is very easy to prove. For $d=3$ \cite[Proposition
1.18]{BCV} implies that $L_{(3t-1,1)}V$ occurs in
$L_{(2,1)}(\Bw^tV)$, so we are done in this case by Proposition
\ref{single_char} (ii). Therefore assume $d>3$.

If $t$ is odd, then $L_{(dt)}V$ occurs in $L_{(d)}(\Bw^tV)$: In
fact, $L_{(dt)}V$ has multiplicity $1$ in $\Tensor \Bw^tV$, so
it can occur in $L_{\mu}(\Bw^tV)$ only if $\mu=(d)$ (the $d$-th
exterior power) or $\mu=(1^d)$ (the $d$-th symmetric power).
Furthermore $(2t)$ is a $t$-predecessor of $(dt)$, and
$\Bw^{2t}V$ occurs in $\Bw^2(\Bw^tV)$ (for instance see
\cite[Lemma 2.1]{BCV}). Therefore $L_{(dt)}V$ occurs in
$\Bw^2(\Bw^tV)\tensor(\Tensor^{d-2}\Bw^tV)$. In particular, it
cannot occur in $L_{(1^d)}(\Bw^tV)$. In the same way, one sees
that $L_{(dt)}V$ occurs in $L_{(1^d)}(\Bw^tV)$ whenever $t$ is
even.

From now on let us assume $t$ odd; the even case is similar. If
$0<k<d-1$, then $(dt-k,1^k)$ has two $(t,1)$-predecessors,
namely
$$
((d-1)t-k,1^k)\qquad\text{and}\qquad((d-1)t-k+1,1^{k-1}).
$$
By induction, the respective Schur modules occur in
$$
L_{(d-k-1,1^k)}\Bigl(\Bw^tV\Bigr)\qquad\text{and}\qquad
L_{(d-k,1^{k-1})}\Bigl(\Bw^tV\Bigr).
$$
So, the Schur modules corresponding to the $(t,1)$-successors
of $((d-1)t-k,1^k)$ can occur in $L_{(d-k,1^k)}(\Bw^tV)$ or in
$L_{(d-k+1,1^{k-1})}(\Bw^tV)$, and the ones corresponding to
the $(t,1)$-successors of $((d-1)t-k+1,1^{k-1})$ can occur in
$L_{(d-k+1,1^{k-1})}(\Bw^tV)$ or in $L_{(d-k,1^k)}(\Bw^tV)$. By
counting multiplicities and using $d>3$, one can check that the
only possibility is that $L_{(dt-k,1^k)}V$ occurs in
$L_{(d-k,1^k)}(\Bw^tV)$. Notice that the multiplicity of
$L_{(dt-k,1^k)}V$ is the same of the one of
$L_{(d-k,1^k)}(\Bw^tV)$ in $\Tensor \Bw^tV$, i.e.
$\binom{d-1}{k}$, so Proposition \ref{single_char} (ii) lets us
conclude.
\end{proof}

We  must pay particular attention to the duals of hooks: The
dual of the hook $(dt-k,1^k)\vdash dt$ is the diagram
$((dt-k)^{d-k-1},(dt-k-1)^k)\vdash d(dt-k-t)$. Notice that is
the unique partition of $d(dt-k-1)$ with $\lambda_d=0$ and
$\lambda_{d-1}\geq \lambda_1-1$.

Before stating the main theorem it is useful to remark the
following:

\begin{lemma}\label{cubics}
A diagram $(a,b,c)\vdash 3t$ (where $c=0$ is not excluded) is
of single $\Bw^t$-type if and only if
$$
\min\{a-b,b-c\}\leq 1.
$$
\end{lemma}

Since all partitions $\lambda\vdash 2t$ are of single
$\Bw^t$-type, one must find exactly those partitions
$(a,b,c)\vdash 3t$ that have no two predecessors in the second
symmetric or second exterior power. Since the latter are easily
characterized (for example, see \cite[Lemma 2.1]{BCV}), the
proof of Lemma \ref{cubics} is an easy exercise. Because of
Proposition \ref{propretract} one may assume $c=0$, and
Proposition \ref{dualization} helps to further reduce the
number of cases.

For the proof of the next theorem we will abbreviate ``single
$\Bw^t$-type" by ``ST" and ``not of single $\Bw^t$-type" by
``NST".

\begin{thm}\label{mainthm}
A $t$-admissible diagram $\lambda=(\lambda_1,\ldots
,\lambda_k)\vdash dt$ is of single $\Bw^t$-type $\mu\vdash d$
if and only if it satisfies one (or more) of the following:
\begin{itemize}
\item[(i)] $\lambda_d\geq t-1$, in which case $\mu =
    (\lambda_1-t+1,\ldots ,\lambda_d-t+1)$.
\item[(ii)] $\lambda_1\leq t+1$, in which case $\mu =
    \lambda^{*,t+1}$.
\item[(iii)] $\lambda_d\geq \lambda_2-1$. If
    $\lambda=(t^d)$, then $\mu=(1^d)$. Otherwise put
    $k=\max\{i:\lambda_i>\lambda_d\}$: according with
    $t-\lambda_d$ being odd or even, $\mu = (d-k, \ 1^{k})$
    or $\mu = (k+1, \ 1^{d-k-1})$.
\item[(iv)] $\lambda_{d-1}\geq \lambda_1-1$. If
    $\lambda=(t^d)$, then $\mu=(1^d)$. Otherwise put
    $k=\min\{i:\lambda_i<\lambda_1\}$: according with
    $\lambda_1-t$ being odd or even, $\mu = (k, \ 1^{d-k})$
    or $\mu = (d-k+1, \ 1^{k-1})$.
\end{itemize}
\end{thm}

If $\lambda$ is in one of the four classes above, then we know
that it is of single $\Bw^t$-type from what done until now: (i)
If $\lambda_d\geq t-1$, then it is a trivial extension of $\mu
= (\lambda_1-t+1,\lambda_2-t+1,\ldots ,\lambda_d-t+1)\vdash d$,
that is obviously of single $\Bw^1$-type; (ii) if
$\lambda_1\leq t+1$, then $\mu = \lambda^{*,t+1}\vdash d$ is of
single $\Bw^1$-type, so Proposition \ref{dualization} let us
conclude; (iii) If $\lambda_d\geq \lambda_2-1$, then $\lambda$
is a trivial extension of a hook. The shape of $\mu$ follows
from Proposition \ref{propretract} and Lemma \ref{hooklemma};
(iv) if $\lambda_{d-1}\geq \lambda_1-1$, then $\lambda^*$ is a
hook. From this, combining Lemma \ref{hooklemma} and
Proposition \ref{dualization}, we get the shape of $\mu$.

As we have just seen, the four classes can described as
follows: (i) consists of the trivial extensions of
$1$-admissible partitions, (ii) is dual to (i) in the sense of
Proposition \ref{dualization}, (iii) contains the hooks and
their trivial extensions, and (iv) is dual to (iii).

The classification in the theorem completely covers the cases
$d=1$ and $d=2$, in which all shapes are of single
$\Bw^t$-type, and also the case $d=3$ done in Lemma
\ref{cubics}. Therefore we may assume that $d\ge 4$. Then the
theorem follows from the next lemma and Proposition
\ref{single_char}. In its proof we will use the theorem
inductively.

\begin{lemma}
If $d\ge 4$ and $\lambda$ is not of one of the types in the
theorem, then it has an NST $(t,1)$-predecessor.
\end{lemma}

The lemma shows that the critical degree is $d=3$ in which the
condition that the predecessors of $\lambda$ occur in pairwise
different predecessors of $\mu$ must be used.

\begin{proof}
If $t=1$ all partitions $\lambda$ fall into the class (i) and
are certainly ST. So we can assume $t\geq 2$.

Suppose first that $\lambda$ is itself a trivial extension.
Then we pass to its trivial reduction $\lambda'$. It is enough
to find an  NST predecessor for $\lambda'$. It yields an NST
predecessor of $\lambda$ after trivial extension. From now on
we can assume that $\lambda$ has at most $d-1$ rows.

Suppose first that $\lambda=(\lambda_1,\dots,\lambda_k)$ is a
successor of a hook. Let $k'=\max\{2,k-1\}$. We choose
$\gamma=((d-1)t-k',2,1^{k'-2})\vdash(d-1)t$. Then $\gamma$ does
not fall into one of the classes (i)--(iv), provided
$\gamma_1\ge t+2$. Using $k'\le d-2$, one derives this
immediately from $d\ge 4$ and $t\geq 2$. The inequality
$\gamma_1\ge t+1$ is sufficient to make $\gamma$ a predecessor
of $\lambda$.

Next suppose $\lambda=(\lambda_1,\lambda_2,1^{k-2})\vdash dt$.
If $\lambda$ has a hook predecessor, then we are done by the
previous case. Therefore we can assume that $\lambda_2\ge t+2$.
If $k=2$, we pass to $\gamma=(\lambda_1,\lambda_2-t)$, and if
$k\ge 3$, we choose
$\gamma=(\lambda_1,\lambda_2-(t-1),1^{k-3})$. Then $\gamma$ is
not of types (i)--(iv). (We are dealing with this case
separately since the duals will come up below.)

In the  remaining case  we choose the predecessor $\gamma$ of
$\lambda$  with the lexicographic smallest set of indices  for
the columns in  which $\gamma$ and $\lambda$ differ by 1. If
$\gamma$ is a hook, then we are done as above. So we can assume
that $\gamma$ is not a hook.

Suppose that $\gamma_1<\lambda_1$. Then $\lambda_2\le t-1$, and
$\gamma$ is not a trivial extension since the bottom row of
$\lambda$ has been removed completely, and $\gamma$ has at most
$d-2$ rows. On the other hand, $\lambda_1+(d-2)\lambda_2\ge dt$
implies $\lambda_1\ge 2t+2$, and so $\gamma_1\ge t+2$. It
follows that $\gamma_{d-2}\le\lambda_{d-2}<\gamma_1-1$, and
$\gamma$ is not of type (i)--(iv).

The case $\gamma_1=\lambda_1>t+1$ remains. We can assume that
$\gamma$ is ST. This is only possible if (1)
$\gamma_{d-2}\ge\gamma_1-1$ or (2) $\gamma$ is the trivial
extension of a hook or (3) $\gamma_{d-1}\ge t-1$.

(1) If $\gamma_{d-2}\ge \gamma_1-1$, then $\lambda_{d-2}\ge
\lambda_1-1$, and $\lambda^*$ is of the second type discussed.
We find an NST predecessor of $\lambda^*$ and dualize back.

(2) If $\gamma$ is a trivial extension of a hook, then
$\gamma_2\le\gamma_{d-1}+1$ and $\lambda_{d-1}\ge t+1$. In
particular $\gamma_2=\lambda_2$, and
$\gamma_{d-1}=\lambda_{d-1}-t\le\lambda_2-t=\gamma_2-t$, which
is a contradiction since $t\geq 2$.

(3) In this case we must have $\lambda_{d-1}\ge 2t-1$ since we
remove $\min\{\lambda_{d-1},t\}$ boxes from row $d-1$ of
$\lambda$. This is evidently impossible (because $t\geq 2$ and
$d\geq 4$).
\end{proof}

\section{Minimal relations of single exterior type}

In this last section we are going to prove the result which
motivated us for producing this note. We will adopt here the
notation given in the introduction.

Let us first recall a result of \cite{BCV}. As already
mentioned, a decomposition of $S_t=\Sym(E\tensor F^*)$ in
irreducible $H$-representations is provided by the Cauchy
formula \eqref{Cauchy}, namely
$$
S_t=\bigoplus_{\mu}L_{\mu}E\tensor L_{\mu}F^*,
$$
where $\mu$ ranges among all the partitions. So, because $G$ is
a subgroup of $H$ whose action is the restriction of that of
$H$, the irreducible $G$-representation $L_{\gamma}V\tensor
L_{\lambda}W^*$ occurs in the $G$-decomposition of $S_t$ if and
only if there exists $\mu\vdash d$ such that $L_{\gamma}V$
occurs in the $\GL(V)$-decomposition of $L_{\mu}(\Bw^tV)$ and
$L_{\lambda}W^*$ occurs in the $\GL(W)$-decomposition of
$L_{\mu}(\Bw^tW^*)$. Moreover, if such a $\mu\vdash d$ exists
and $\gamma$ and $\lambda$ are both of single $\Bw^t$-type,
then $L_{\gamma}V\tensor L_{\lambda}W^*$ is a direct summand of
$J_t\tensor_{S_t}K$ if and only if $\gamma\neq \lambda$ and the
predecessors of $\gamma$ and of $\lambda$ coincide
\cite[Proposition 1.21 and Theorem 1.23(iv)]{BCV}. This is the
fact on which the proof of the next theorem is based.

\begin{thm}\label{singleTyperel}
Let $L_{\gamma}V\tensor L_{\lambda}W^*$ be a direct summand of
$J_t\tensor_{S_t}K$ such that both $\gamma$ and $\lambda$ are
diagrams of single $\Bw^t$-type. Then $(\gamma|\lambda)\in
A\cup B$.
\end{thm}

\begin{proof}
For $t=1$ there is nothing to prove because $J_1=(0)$. So
assume $t\geq 2$.

From what said above $\gamma$ and $\lambda$ must be
$t$-admissible partitions of the same number $dt$. If $d=1$
then $\gamma=\lambda=(t)$; if $d=2$, then $(J_t)_2\cong
\bigoplus_{(\gamma|\lambda)\in A}L_{\gamma}V\tensor
L_{\lambda}W^*$  by \cite[Lemma 2.1]{BCV}; if $d=3$ then
\cite[Proposition 3.16]{BCV} does the job.

So from now on we will focus on $d\geq 4$. Recall that in
Theorem \ref{mainthm} have been identified 4 (not disjoint)
sets of diagrams, say $E_1^t=\{\mbox{diagrams as in (i)}\}$,
$E_2^t=\{\mbox{diagrams as in (ii)}\}$ and so on, such that:
$$
\{\mbox{diagrams of single $\Bw^t$-type}\}=
E_1^t\cup E_2^t\cup E_3^t\cup E_4^t.
$$

We start by showing the following:
\smallskip
\begin{lemma}\label{sameclass}
Let $L_{\gamma}V\tensor L_{\lambda}W^*$ be a direct summand of
$J_t\tensor_{S_t}K$ such that both $\gamma$ and $\lambda$ are
diagrams of single $\Bw^t$-type belonging to the same $E_i^t$
for some $i\in\{1,2,3,4\}$. Then $(\gamma|\lambda)\in A\cup B$.
\end{lemma}

\begin{proof}[Proof of Lemma \ref{sameclass}]
We know that $\gamma$ and $\lambda$ are different
$t$-admissible partitions of $dt$ sharing the same $\mu\vdash
d$. This excludes $i\in \{1,2\}$, because in these cases
Theorem \ref{mainthm} says that $\gamma$ and $\lambda$ cannot
share the same $\mu$ if they are different.

Suppose $i=3$. We must have $\gamma_d\neq \lambda_d$ if
$\lambda$ and $\gamma$ belong to the same $\mu$. Assume
$\lambda_d > \gamma_d$. The diagram $\gamma$ has a predecessor
$\gamma'$ with $\gamma'_{d-1}=\gamma_d$. This cannot be a
predecessor of $\lambda$, and so $\gamma$ and $\lambda$ do not
have the same predecessors.

So only the case $i=4$ remains. Let
$s=\max\{\gamma_1,\lambda_1\}$. If $s=t$, then $\gamma =
\lambda = (t^d)$, so we can assume $s>t$. If $\gamma$ and
$\lambda$ share the same $\mu$, by combining Propositions
\ref{dualization} and \ref{propretract} $\gamma^{*,s}$ and
$\lambda^{*,s}$ share $\mu$ as well. Of course $\gamma^{*,s}$
and $\lambda^{*,s}$ belong to $E_3^{s-t}$, and they are
different if $\gamma$ and $\lambda$ are different. In this
case,  we know by the previous case that $\gamma^{*,s}$ and
$\lambda^{*,s}$ have different predecessors, and by dualizing
we infer the same for $\gamma$ and $\lambda$.
\end{proof}

Let us go ahead with the proof of Theorem \ref{singleTyperel}.
Set $\gamma=(\gamma_1,\ldots ,\gamma_h)$ and
$\lambda=(\lambda_1,\ldots ,\lambda_k)$ with $h,k\leq d$. If
$h=k=d$ we can use induction on $t$ since both $\gamma$ and
$\lambda$ are trivial extensions.

If $h$ and $k$ are both less than $d$, then neither $\gamma$
nor $\lambda$ belong to $E_1^t$. Assume that $\gamma \in
E_3^t$. Then $\mu$ is a hook. By the lemma, $\lambda\notin
E_3^t$. Since $\mu$ is a hook, $k<d$ and $\lambda \notin
E_3^t$, it follows that $\lambda\in E_4^t$ (recall that $E_2^t$
and $E_4^t$ are not disjoint). Because $\gamma$ is a
$t$-admissible hook and $h<d$, we get $\gamma_1>dt-d+1\geq
4t-3$. Then $\lambda_1>3t-3$, otherwise $\gamma$ and $\lambda$
would have different predecessors. Therefore $\lambda\vdash dt
> (d-1)(3t-3)$, which is impossible whenever $d\geq 3$ (recall
that $t\geq 2$). So, by symmetry, we can assume that neither
$\gamma$ nor $\lambda$ is in $E_3^t$. Therefore $\gamma$ and
$\lambda$ belong to $E_2^t\cup E_4^t$. However $\gamma$ and
$\lambda$ share the same $\mu$ and, in such a situation, $\mu$
is a hook if and only if $\gamma$ and $\lambda$ both belong to
$E_4^t$, a case already excluded in the lemma.

So, we can assume by symmetry that $h<d$ and $k=d$. Notice that
$h=d-1$, because all the predecessors of $\lambda$ will have
$d-1$ rows. For the same reason we can even infer that
$\gamma_{d-1}>t$, otherwise we could entirely remove
$\gamma_{d-1}$ getting a predecessor of $\gamma$ with $d-2$
rows. Since $d\geq 4$, we have $\gamma_2'>t$ for all $\gamma'$
predecessors of $\gamma$. So $\lambda$ does not belong to
$E_3^t$, since in this case $\lambda_2\leq t$. Since
$\gamma_{d-1}>t$, Theorem \ref{mainthm} tells us that
$\gamma\in E_4^t$ (once again, recall that $E_2^t$ and $E_4^t$
are not disjoint): so $\mu$ must be a hook. If $\lambda\in
E_1^t$, then $\gamma_{d-1}\geq 2t-1$ (otherwise $\gamma$ would
have a predecessor $\gamma'$ with $\gamma'_{d-1}<t-1$, that
cannot be a predecessor of $\lambda$). This is evidently
impossible if $d\geq 4$. So Theorem \ref{mainthm} implies that
$\lambda\in E_4^t$, and the lemma lets us conclude.

\end{proof}


\begin{thebibliography}{99}
\addcontentsline{toc}{section}{Bibliography}
\bibitem[BCV]{BCV} W. Bruns, A. Conca, M. Varbaro, {\it
    Relations among the Minors of a Generic Matrix}. Adv. Math. 244, pp. 171-206, 2013.
\bibitem[DEP]{DEP} C. De Concini, D. Eisenbud, C. Procesi,
    \textit{Young Diagrams and Determinantal Varieties},
    Invent. Math. 56, pp. 129-165, 1980.
\bibitem[FH]{FH} W. Fulton, J. Harris, \emph{Representation
    Theory. A First Course}, Graduate Texts in Mathematics 129,
    1991.
\bibitem[We]{We} J. M. Weyman, \textit{Cohomology of vector
    bundles and syzygies}, Cambridge Tracts in Mathematics 149,
    2003.
\end{thebibliography}
\end{document}